\documentclass[12pt]{amsart}

\usepackage{fullpage}
\usepackage{amsmath}
\usepackage{amssymb}
\usepackage{amsthm}
\usepackage{amstext}
\usepackage{appendix}
\usepackage{perpage}
\usepackage[x11names]{xcolor}
\usepackage{tikz-cd}
\usepackage{ytableau}\ytableausetup{centertableaux}
\usepackage[colorlinks=true, allcolors=blue]{hyperref}
\usepackage{todonotes}

\newtheorem{thm}{Theorem}[section]
\newtheorem{lemm}[thm]{Lemma}
\newtheorem{coro}[thm]{Corollary}

\theoremstyle{definition}

\begin{document}

\title{A remark on decomposing the canonical representation of the Drinfeld curve}

\author{Zhe Chen \and Yushan Pan}

\address{Department of Mathematics, Shantou University, Shantou, 515821, China}
\email{zhechencz@gmail.com}

\address{}
\email{panys@mail.sustech.edu.cn}

\begin{abstract}
Recently, by studying an explicit basis, K\"ock and Laurent give the decomposition of the  $\overline{\mathbb{F}}_q[\mathrm{SL}_2(\mathbb{F}_q)]$-module of holomorphic forms on the Drinfeld curve. We present a crystalline cohomological proof of a weaker version of this result, without specifying a basis. As a by-product we observe a similar decomposition for the Gelfand--Graev representations.
\end{abstract}

\maketitle

\section{Introduction}

Decomposing the space of holomorphic forms on an algebraic curve with respect to a group action, which dates back to the work of Hecke \cite{Hecke1928}, is an intriguing problem.  Over the past decades much progress has been made, but the problem is not yet completely solved; details on the historical developments can be found in \cite{Bleher_Chinburg_Kontogeorgis_2020_Galois_structure} and \cite{koeck-laurent}. 

\vspace{2mm} In this short note we focus on a special case, and we revisit the main result of \cite{koeck-laurent}.

\vspace{2mm} Consider the algebraic group $G=\mathrm{SL}_2$ over $\mathbb{F}_q$. Let $F$ be the geometric Frobenius endomorphism on $G$; so $G^F$ is the finite group $\mathrm{SL}_2(\mathbb{F}_q)$. Let $C/\mathbb{F}_q$ be the projective Drinfeld curve, which is defined by the equation 
$$XY^{q}-X^{q}Y-Z^{q+1}=0.$$ 
Note that $G^F$ acts on $C$ via left matrix multiplications on 
$\begin{bmatrix}
X\\
Y
\end{bmatrix}$. 
We denote by $K$ a finite extension of $\mathrm{Frac}(W(\mathbb{F}_q))$ such that every complex irreducible representation of $G^F$ is realisable over $K$, where $W(\mathbb{F}_q)$ means the ring of Witt vectors over $\mathbb{F}_q$. Let $k$ be the residue field of $K$. We are interested in the $kG^F$-module $H^0(C_k,\Omega^1_{C_k})$ of global holomorphic forms. In literature this module is also referred to as the \emph{canonical representation} of $C_k$. 

\vspace{2mm} If $A$ is one of $K$ or $k$, and if $M$ is an $AG^F$-module, we denote by $[M]_A$ the image in the Grothendieck group $G_0(AG^F)$ of finitely generated $AG^F$-modules. Then the \emph{$G_0$-version} of \cite[Theorem~2.1]{koeck-laurent} can be stated as:
\begin{thm}[Laurent--K\"ock]\label{thm:main}
In $G_0(kG^F)$ we have the decomposition
$$[H^0(C_k,\Omega^1_{C_k})]_k=\sum_{i=0}^{q-2} [V_i]_k,$$
where $V_i:=\mathrm{Sym}^i(k^2)$ with $k^2$ being the natural representation of $G^F\subseteq\mathrm{GL}_2(k)$.
\end{thm}
Actually, \cite{koeck-laurent} proves the much stronger result on the corresponding module decomposition, rather than the above $G_0$-version. The method used in \cite{koeck-laurent} is  to construct an explicit basis of $H^0(C_k,\Omega^1_{C_k})$, and then study the $G^F$-action on it. Our aim is to record a less explicit, but probably more conceptual argument of the above result, by using  crystalline cohomology. Note that similar methods are already carried on in literature, in various situations (see e.g. \cite{HAASTERT199077} and \cite{Hortsch2012}), in which usually a basis of the de Rham cohomology is constructed, and explicit computations are made using \v{C}ech cohomology. Here we will be concerned with a simpler route, without computing a basis or \v{C}ech cohomology. Along the way we find that the Gelfand--Graev representations of $G^F$ naturally appear, and admit a similar decomposition (see Lemma~\ref{lemm:crytalline rep of C and GG rep} and Corollary~\ref{coro:GGGR}).

\vspace{2mm} \noindent {\bf Acknowledgement.} We are grateful to Bernhard K\"ock for several helpful comments, and thank K\"ock and Bonnaf\'e for pointing out a gap in an earlier version of this paper. This work is partially supported  by   Natural Science Foundation of Guangdong No.~2023A1515010561 and NSFC No.~12171297.

\section{Representations and cohomology}

Throughout, we shall use the notations 
$$H^i(-/W),\ H_{dR}^i(-/k),\ H^i(-,\overline{\mathbb{Q}}_{\ell}),\ H_c^i(-,\overline{\mathbb{Q}}_{\ell}),$$
for the crystalline cohomology (a $W$-module, where $W:=W(k)$ is the ring of Witt vectors over $k$), the de Rham cohomology (a $k$-space), the $\ell$-adic cohomology (a $\overline{\mathbb{Q}}_{\ell}$-space), and the compactly supported version of $\ell$-adic cohomology (a $\overline{\mathbb{Q}}_{\ell}$-space), respectively. Here $\ell$ is a prime not equal to $p:=\mathrm{char}(k)$.

\vspace{2mm} Consider the affine plane curve $D$ given by the equation $xy^q-x^qy-1=0$; this is the open affine subset of $C$ at $Z=1$, on which $G^F$ acts as well. Let $T$ be an $F$-stable non-split torus of $G$. Then $T^F\cong\mu_{q+1}\subseteq\mathbb{F}_{q^2}$ acts on $C$ and $D$ via the scalar multiplication on
$\begin{bmatrix}
X\\
Y
\end{bmatrix}$. 
We will use the knowledge of Deligne--Lusztig (virtual) representations of $G^F$ attached to $T$. These representations are defined to be
$$R^{\theta}:=\sum_i(-1)^iH^i_c(D,\overline{\mathbb{Q}}_{\ell})_{\theta},$$
where $\theta\in\mathrm{Irr}(T^F)$ and the subscript $(-)_{\theta}$ means we are taking the $\theta$-isotypical part. They are all of dimension $1-q$, and unless $\theta=1$, we have $R^{\theta}=-H^1_c(D,\overline{\mathbb{Q}}_{\ell})_{\theta}$ (for the details we refer to \cite{Bonnafe_2011_rep_SL2_book}); let us denote the non-trivial $\theta$'s by $\theta_1,\cdots,\theta_q$.

\begin{lemm}\label{lemma:l-adic rep of C}
As representations of $G^F$, we have 
$$H^1(C,\overline{\mathbb{Q}}_{\ell})\cong-\sum_{i=1}^{q} R^{\theta_i}
\quad\textrm{and}\quad
H^2(C,\overline{\mathbb{Q}}_{\ell})=H^0(C,\overline{\mathbb{Q}}_{\ell})=1_{G^F}.$$ 
\end{lemm}
\begin{proof}
Consider the long exact sequence associated with the partition $C=D\sqcup (C\setminus D)$:
\begin{equation*}
\begin{split}
0 &\rightarrow H_c^0(D,\overline{\mathbb{Q}}_{\ell}) \rightarrow H^0(C,\overline{\mathbb{Q}}_{\ell}) \rightarrow H^0(C\setminus D,\overline{\mathbb{Q}}_{\ell})\rightarrow H_c^1(D,\overline{\mathbb{Q}}_{\ell}) \rightarrow H^1(C,\overline{\mathbb{Q}}_{\ell}) \\
& \rightarrow H^1(C\setminus D,\overline{\mathbb{Q}}_{\ell}) \rightarrow H_c^2(D,\overline{\mathbb{Q}}_{\ell}) \rightarrow H^2(C,\overline{\mathbb{Q}}_{\ell}) \rightarrow H^2(C\setminus D,\overline{\mathbb{Q}}_{\ell}) \rightarrow \ldots
\end{split}
\end{equation*}
Since $D$ is an affine curve and $C\setminus D$ is a finite set, this exact sequence splits to be
$$0\longrightarrow H^0(C,\overline{\mathbb{Q}}_{\ell})\longrightarrow H^0(C\setminus D,\overline{\mathbb{Q}}_{\ell})\longrightarrow \mathrm{St}-\sum_iR^{\theta_i}\longrightarrow H^1(C,\overline{\mathbb{Q}}_{\ell})\longrightarrow 0$$
and
$$0\longrightarrow 1_{G^F} \longrightarrow H^2(C,\overline{\mathbb{Q}}_{\ell})\longrightarrow0,$$
where $\mathrm{St}$ denotes the Steinberg representation. Note that the $G^F$-action on $C\setminus D$ can be viewed as the natural $G^F$-action on $\mathbb{P}^1(\mathbb{F}_q)$, which corresponds to the representation $1_{G^F}+\mathrm{St}$. Now the assertion follows from a simple calculation.
\end{proof}

\begin{lemm}\label{lemm:crytalline rep of C}
One has
$$H^1(C_k/W)\otimes_W\overline{K}
\cong
-\sum_{i=1}^{q} R^{\theta_i},$$
as representations of $G^F$, after specifying an isomorphism of abstract fields $\overline{K}\cong \overline{\mathbb{Q}}_{\ell}$. 
\end{lemm}

\begin{proof}
By Lemma~\ref{lemma:l-adic rep of C}, it suffices to show that 
$$H^1(C_k/W)\otimes_W\overline{K}\cong H^1(C,\overline{\mathbb{Q}}_{\ell})$$
as representations of $G^F$; this is the special case of \cite[Theorem~2]{Katz_Messing_1974_Consequences_Riemann_hypothesis} applying to the algebraic cycles  $\Gamma_g\subseteq C\times C$ (the graph of the action of $g\in G^F$).
\end{proof}

Let 
$$d\colon G_0(KG^F)\longrightarrow G_0(kG^F)$$ 
be the decomposition map (\cite[15.2]{serre1977lin_rep_finite_gr}). Then \cite[Proposition~10.2.9(d)]{Bonnafe_2011_rep_SL2_book} tells that:

\begin{lemm}\label{lemma: decomposition map of DL rep}
One can arrange the order of $\theta_1,\cdots,\theta_q$, so that 
$$d([R^{\theta_i}]_K)=-[V_{i-2}]_k-[V_{q-i-1}]_k$$
for every $i=1,\cdots,q$, where $V_{-1}:=0$.
\end{lemm}

\begin{lemm}\label{lemma:crystalline and de Rham k-rep of C}
we have 
$$[H_{dR}^1(C_k/k)]_k=[H^1(C_k/W)\otimes_W k]_k=2\sum_{i=0}^{q-2}[V_i]_k.$$
\end{lemm}
\begin{proof}
Since $C_k$ is a curve, $H^2(C_k/W)\cong W$ (\cite[3.5]{Illusie_1974_1975_Report_CrystallineCohom}), so the first equality follows from the canonical short exact sequence (see \cite[3.4]{Illusie_1974_1975_Report_CrystallineCohom})
$$0\longrightarrow H^1(C_k/W)\otimes_W k\longrightarrow H_{dR}^1(C_k/k)\longrightarrow \mathrm{Tor}^W_{1}(H^{2}(C_k/W),k) \longrightarrow 0.$$
The second equality,  by Lemma~\ref{lemm:crytalline rep of C} and Lemma~\ref{lemma: decomposition map of DL rep}, would follow from that $H^1(C/W)$ has no torsion; the latter assertion holds true for the first crystalline cohomology of any smooth projective variety (see \cite[3.11.2 and 7.1]{Illusie_1979_ANES_deRham_Witt}).
\end{proof}

\begin{proof}[Proof of Theorem~\ref{thm:main}.]
Note that the same equation of $C$ defines a smooth lift of $C$ over $W$, so the Hodge--de Rham spectral sequence of $C_k$ degenerates at the first page (\cite[Corollaire~2.4]{Deligne_Illusie_1987_deRham}), thus by Poincar\'e duality (see e.g.\ \cite[Lemma~50.20.1]{Stacks_Project}) we get
$$[H_{dR}^1(C_k/k)]_k=[H^0(C_k,\Omega_{C_k}^1)]_k+[\overline{H^0(C_k,\Omega_{C_k}^1)}]_k,$$
in which $\overline{(-)}$ denotes the dual representation. So Lemma~\ref{lemma:crystalline and de Rham k-rep of C}  gives that
$$[H^0(C_k,\Omega_{C_k}^1)]_k+[\overline{H^0(C_k,\Omega_{C_k}^1)}]_k=2\sum_{i=0}^{q-2}[V_i]_k.$$
As $\{[V_0]_k,[V_1]_k,\cdots,[V_{q-1}]_k\}$ is a $\mathbb{Z}$-basis of the group $G_0(kG^F)$ \cite[10.2.3]{Bonnafe_2011_rep_SL2_book},  and as each $V_i$ has the same Brauer character with $\overline{V_i}$ by construction, we conclude that
$$[H^0(C_k,\Omega_{C_k}^1)]_k=\sum_{i=0}^{q-2}[V_i]_k,$$
as desired.
\end{proof}

\section{Gelfand--Graev representations}

Let $B$ be the subgroup of $G=\mathrm{SL}_2$ consisting of upper-triangular matrices, and let $U$ be its unipotent radical. Then the diagonal subgroup of $G$, denoted by $S$, acts on $U$ by conjugation. If $q$ is odd, the \emph{non-trivial} irreducible characters of $U^F$ form two $S^F$-orbits; for each of the two orbits we choose a representative $\psi_i$ ($i=1,2$). If $q$ is even, there is only one such $S^F$-orbit, and we use both $\psi_1$ and $\psi_2$ to denote the same chosen representative. Then
$$\Gamma_i:=\mathrm{Ind}_{U^F}^{G^F}\psi_i$$
are the so-called Gelfand--Graev representations of  $G^F=\mathrm{SL}_2(\mathbb{F}_q)$, which are independent of the choice of the representatives $\psi_i$ and are multiplicity-free. Gelfand--Graev representations play an important role in the representation theory of finite groups of Lie type, and we refer to \cite[12.3]{DM_book_2nd_edition} for more details.

\vspace{2mm} Combining Lemma~\ref{lemm:crytalline rep of C} and $\Gamma_i$ we observe a simple algebraic realisation of the cohomological representation $H^1(C_k/W)\otimes_W\overline{K}$:

\begin{lemm}\label{lemm:crytalline rep of C and GG rep}
One has
$$H^1(C_k/W)\otimes_W\overline{K}
\cong \Gamma_1+\Gamma_2-\mathrm{Ind}_{U^F}^{G^F}1_{U^F}-\mathrm{Ind}_{B^F}^{G^F}1_{B^F}+2\cdot1_{G^F}.$$
\end{lemm}
\begin{proof}
Note that the character values of $\Gamma_1+\Gamma_2$ can be easily computed using the formula for induced characters: At non-unipotent elements the character values are zero, and at regular unipotent elements the character values are  $-2$. Meanwhile, the character values of $R^{\theta_i}$ can be found, e.g.\ in \cite[Table~5.4]{Bonnafe_2011_rep_SL2_book} (this table is for odd $q$; for even $q$ the similar table holds after removing the representations of dimension $(q\pm1)/2$). Then the assertion follows from Lemma~\ref{lemm:crytalline rep of C} and a direct comparison of character values.
\end{proof}

\begin{coro}\label{coro:GGGR}
In $G_0(kG^F)$ we have
$$d([\Gamma_1]_K)=d([\Gamma_2]_K)=2\sum_{i=1}^{q-2} [V_i]_k+[V_0]_k+[V_{q-1}]_k.$$
\end{coro}
\begin{proof}
By Lemma~\ref{lemma:crystalline and de Rham k-rep of C}, Lemma~\ref{lemm:crytalline rep of C and GG rep}, and \cite[10.2.9(c)]{Bonnafe_2011_rep_SL2_book} we see that
$$d([\Gamma_1]_K+[\Gamma_2]_K - [\mathrm{Ind}_{U^F}^{G^F}1_{U^F}]_K)=2\sum_{i=1}^{q-2} [V_i]_k+[V_0]_k+[V_{q-1}]_k.$$
Since $U^F$ is a Sylow $p$-subgroup of $G^F$, the Brauer characters of $\Gamma_1$, $\Gamma_2$, and $\mathrm{Ind}_{U^F}^{G^F}1_{U^F}$ are the same by construction, which implies that $d([\Gamma_1]_K)=d([\Gamma_2]_K)=d([\mathrm{Ind}_{U^F}^{G^F}1_{U^F}]_K)$ (\cite[18.2]{serre1977lin_rep_finite_gr}). Hence the assertion holds.
\end{proof}

Alternatively, one can also deduce this corollary using the observation $[kG^F]_k=q\cdot d([\Gamma_i]_K)$ together with the knowledge of the character table of $G^F$ and full of \cite[10.2.9]{Bonnafe_2011_rep_SL2_book}.

\bibliographystyle{alpha}
\bibliography{zchenrefs}

\end{document}